\documentclass{amsart}

\usepackage{latexsym,amsthm,amssymb}

\DeclareMathAlphabet{\mathpzc}{OT1}{pzc}{m}{it}

\usepackage[leqno]{amsmath}

\usepackage{mathrsfs}

\usepackage[all]{xy}

\usepackage{color}
\usepackage{colordvi}

\usepackage{verbatim}

\usepackage{ifthen}
\newboolean{hide.proofs}
\setboolean{hide.proofs}{false}

\newtheorem{theorem}{Theorem}[section]

\newtheorem{proposition}[theorem]{Proposition}

\newtheorem{remark}[theorem]{Remark}

\newtheorem{lemma-conjecture}[theorem]{Lemma--Conjecture}

\numberwithin{equation}{theorem}

\renewcommand{\mathcal}{\mathscr}

\renewcommand{\mathbb}{\mathbf}

\title{Deformations of canonical triple covers}

\author{Francisco Javier Gallego}
\author{Miguel Gonz\'alez}
\author{\\ Bangere P. Purnaprajna}

\address{Departamento de \'Algebra, Universidad Complutense de Madrid and Instituto de Matem\'atica Interdisciplinar, Madrid, Spain}
\email{gallego@mat.ucm.es}
\address{Departamento de \'Algebra, Universidad Complutense de Madrid, Madrid, Spain}
\email{mgonza@mat.ucm.es}
\address{Department of Mathematics, University of Kansas, Lawrence, Kansas, USA }
\email{purna@math.ku.edu}

\begin{document}

\begin{abstract}
In this paper, we show that if $X$ is a smooth variety of general type of dimension $m \geq 3$ for which the canonical map induces a triple cover onto $Y$, where $Y$ is a projective bundle over $\mathbf P^1$ or onto a projective space or onto a quadric hypersurface, embedded by a complete linear series, then the general deformation of the canonical morphism of $X$ is again canonical and induces a triple cover.  The extremal case when $Y$ is embedded as a variety of minimal degree is of interest, due to its appearance in numerous situations. This is especially interesting as well, since it has no lower dimensional analogues. 
\end{abstract}

\maketitle

\section*{Introduction}
           In this article we prove new results on deformations of canonical morphisms of degree $3$ of a variety of general type of arbitrary dimension when the image $Y$ of the morphism is a rational variety such as the projective space, a quadric hypersurface or a $\mathbf P^n$ bundle over a projective line. We give special emphasis on those $Y$ when it is a variety of minimal degree. This is a new case that does not occur in either curves or surfaces. Note that the degree of the canonical morphism for curves is bounded by $2$. It was shown in \cite{Seshadri} that there are no odd degree canonical covers of smooth surfaces of minimal degree. The geometric genus of a triple canonical cover of a singular surface of minimal degree is bounded by 5. So the geometry of the higher dimensional covers has no analogy in lower dimensions. In this article we handle more general triple covers even when $Y$ is not of minimal degree. It is a well known fact that low degree covers of varieties of minimal degree have a ubiquitous presence in the geometry of varieties of general type, especially in the case of algebraic surfaces and threefolds such as Calabi-Yau. They appear in the geometry of higher dimensional varieties of general type as well, as \cite{Fujita1} and \cite{Kobayashi} and the very recent work of \cite{CPG} show. These deal mostly with the case of double covers. The work \cite{Fujita2} deals with structure theorem for triple covers.

\medskip

      The result of Castelnuovo for an algebraic surface of general type says that if the linear system $|K_{X}|$ is birational, then $K_{X}^{2}\geq 3p_g-7$. There is no known general result for a threefold of general type in these directions. The triple canonical covers of rational varieties satisfy the inequality $K_{X}^{3}\geq 3p_g-9$. There is a strong correlation in the numerology between surfaces and threefolds.  It was shown in a very recent work in \cite{CPG} that the geometry of Horikawa threefolds, which satisfy $K_{X}^{3}=2p_g-6$ has striking similarities to Horikawa surfaces, that is those surfaces which are on the Noether line $K_{X}^{2}=2p_g-4$.  This analogy is important for understanding the geometry of the moduli spaces of higher dimensional varieties. The numerology suggests that for a threefold, the analogy of Castelnuovo's bound could be $K_{X}^{3}\geq 3p_g-9$. That is for a threefold $X$, if the linear system $|K_{X}|$  induces a birational map, then $K_{X}^{3}\geq 3p_g-9$. In any case, the equality $K_{X}^{3}=3p_g-9$, is sufficiently high for a threefold, and $|K_X|$ can be potentially (or conjecturally) birational. This is just for the case when the image of the canonical map $Y$ is of minimal degree. If the image is not of minimal degree, then $K_{X}^{3}$ can be much higher.  But the result in this article shows that even when $K_{X}^{3}$ is sufficiently high, as in the case of triple covers, the deformation of the the degree $3$ morphisms are again degree $3$. The impact of these results on the moduli components is clear. For the case of algebraic surfaces, the analogues of the results proved here for higher dimensions, showed such an impact on the moduli components, as demonstrated in \cite{MP}.
      
\medskip

It is known from results in \cite{Seshadri}, that there are no triple canonical covers of a smooth variety of minimal degree if the dimension of the variety of general type is even. In this article we do show that for an algebraic surface of general type, there are no triple canonical covers of rational scrolls $Y$ or $\mathbf P^2$ or a quadric $\mathbf Q_{2}$, whether they are embedded as a variety of minimal degree or not.  But if the dimension of $X$ is odd, there are infinite families of examples of triple canonical covers, and this appears in the last section. 

\smallskip

{\bf Acknowledgements}: The first and the second author were partially supported by grants MTM2009-06964 and MTM2012-32670 and by the UCM research group 910772. The first author also thanks the Department of Mathematics of the University of Kansas for its hospitality. The third author thanks the National Science foundation (Grant Number: 1206434) for the partial support of this work.

\section{Deformations of a canonical triple cover of a projective bundle over $\mathbf P^1$ or a projective space or a quadric hypersurface embedded by a complete linear series}
We work over $\mathbb C$. We will use the following 

\noindent
{\bf Set--up and notation:} 
\begin{enumerate}
\item $X$ is a smooth variety of general type of dimension $m \geq 3$
with ample and base--point--free canonical bundle and canonical map $\varphi :X \to \mathbb P ^N$
of degree $3$. 
\item The image $Y$ of $\varphi$ is one of the following varieties:
	\begin{enumerate}
		\item a smooth projective bundle $Y=\mathbb P(E)$ (of dimension $m$) over $\mathbf P^1$ (where $E=\mathcal O_{\mathbb P^1} \oplus \mathcal O_{\mathbb P^1}(-e_1) \oplus \cdots \oplus \mathcal O_{\mathbb P^1}(-e_{m-1})$ with $0 \leq e_1\leq \cdots \leq e_{m-1}$) embedded, as a non--degenerate variety, in $\mathbb P ^N $ by the complete linear series of a very ample divisor $aT+bF$, for which $T$ is the divisor on $Y$ such
		that $\mathcal O_Y(T) = \mathcal O_{\mathbf P(E)}(1)$ and $F=\mathbb P^{m-1}$ is a fiber, or
		\item projective space $Y=\mathbb P^m$ embedded, as a non--degenerate variety, in $\mathbb P^N$ by the complete linear series of $\mathcal O_{\mathbb P^m}(d)$, or
		\item a smooth quadric hypersurface $Y=\mathbb Q_m$ in $\mathbb P^{m+1}$ embedded, as a non--degenerate variety, in $\mathbb P^N$ by the complete linear series of $\mathcal O_{\mathbb Q_m}(d)$, for which $m+d \geq 6$.
	\end{enumerate}
\item Let $i: Y \hookrightarrow \mathbb P^N$ denote the embedding and let $\pi: X \to Y$ denote the induced triple cover such that $\varphi=i \circ \pi$.
\end{enumerate}
\begin{remark}\label{Seshadri.remark}{\rm Under the conditions of our set--up, $\varphi$ factors through a triple
cover $\pi: X \longrightarrow Y$ whose trace--zero module is a
vector bundle $\mathcal E$ of rank $2$. In addition, \cite[Theorem
4.3]{Seshadri} says that $\varphi$ is the canonical map of $X$ if, and only if,
	\begin{enumerate}
	\item the triple cover $\pi$ is cyclic with $\mathcal E = \frac{1}{2}(\omega_Y(-1)) \oplus \omega_Y(-1)$, and 
	\item  $H^0(\mathcal O_Y(1) \otimes \frac{1}{2}(\omega_Y(-1)))=0$.
   \end{enumerate}
   Observe that, since $p_g(Y)=0$, condition $(2)$ is equivalent to the fact that the pullback of the complete linear series of $\mathcal O_Y(1)$ is the complete canonical series of $X$.}
\end{remark}
\begin{remark}\label{-3/2}
{\rm Under the conditions of our set--up, there exists a triple cover $\pi$ for which $\varphi= i \circ \pi$ is the canonical map of $X$ if, and only if,
	\begin{enumerate}
	\item  the linear system $|\frac{-3}{2}(\omega_Y(-1))|$ contains a smooth member, and
	\item $H^0(\mathcal O_Y(1) \otimes \frac{1}{2}(\omega_Y(-1)))=0$.
	\end{enumerate}}
\end{remark}
\begin{remark}\label{numerical}
{\rm Under the conditions of our set--up, numerical conditions which are sufficient for $\pi$ to exist are:
\begin{enumerate}
\item In the case $Y=\mathbb P(E)$
	\begin{enumerate}
		\item $a \geq 1$ and $b \geq a e_{m-1}+1$ (for $\mathcal O_Y(1)=\mathcal O_{\mathbb P(E)}(aT+bF)$ very ample), and
		\item $m+a$ is even, and $b+e_1+\cdots+e_{m-1}$ is even (for $\frac{1}{2}(\omega_Y(-1)))$ to exist), and
		\item $b \geq -(e_1+\cdots+e_{m-2}+(1-m-a)e_{m-1}+2)$ (for $|\frac{-3}{2}(\omega_Y(-1))|$ base--point--free), and
		\item $a \leq m-1$ (for $H^0(\mathcal O_Y(1) \otimes \frac{1}{2}(\omega_Y(-1)))=0$).
		\end{enumerate}
\item In the case $Y=\mathbb P^m$
	\begin{enumerate}
	\item $d \geq 1$ (for $\mathcal O_Y(1)=\mathcal O_{\mathbb P^m}(d)$ very ample), and
	\item $m+d$ is odd (for $\frac{1}{2}(\omega_Y(-1)))$ to exist), and
	\item $d \leq m$ (for $H^0(\mathcal O_Y(1) \otimes \frac{1}{2}(\omega_Y(-1)))=0$). 
	\end{enumerate}
\item In the case $Y=\mathbb Q_m$
	\begin{enumerate}
	\item $d \geq 1$ (for $\mathcal O_Y(1)=\mathcal O_{\mathbb Q_m}(d)$ very ample), and
		\item $m+d$ is odd (for $\frac{1}{2}(\omega_Y(-1)))$ to exist), and
		\item $d \leq m-2$ (for $H^0(\mathcal O_Y(1) \otimes \frac{1}{2}(\omega_Y(-1)))=0$). 
	\end{enumerate}
\end{enumerate}}
\end{remark}

\noindent
In \cite{Seshadri}, a result there shows that that there are no triple canonical covers of a smooth variety of minimal degree if the dimension of the variety of general type is even. A natural question would be to know if there are any canonical triple covers at all of say projective bundle $Y$ over a line, such that it is embedded by the complete canonical linear series. 
The next result shows that for an algebraic surface of general type, there are no triple canonical covers of rational scrolls $Y$ or $\mathbf P^2$ or a quadric $\mathbf Q_{2}$ per se, whether they are embedded as a variety of minimal degree or not.  If we relax the condition that it is only the sublinear series of the canonical linear series, but not the whole linear series, then the following result does tell us that such exist for the class of algebraic surfaces. 

\begin{proposition}
In the case $m=2$, of a Hirzebruch surface $Y=\mathbb F_e$ or $Y=\mathbb P^2$ or $Y=\mathbb Q_2$, there exist cyclic triple covers $\pi: X \to Y$ for which $\omega_X=\pi^{\ast}(\mathcal O_Y(1))$, but the pullback $\pi^{\ast}H^0(\mathcal O_Y(1))$ is a non--complete subseries of $H^0(\omega_X)$, except for the case $Y=\mathbb P ^2$ and $d=1$.
\end{proposition}
\begin{proof}
We will use the conditions of Remark~\eqref{-3/2}.

\noindent
First we deal with the case $Y=\mathbb F_e$.
We claim that $|\frac{-3}{2}(\omega_Y(-1))|$ contains a smooth member if, and only if, $(1)$ $e \leq 3$, or $(2)$ ($b-ae \geq e - 2$ or $b-ae = \frac{1}{3}e-2$), if $e\geq 4$. Indeed, if $|\frac{-3}{2}(\omega_Y(-1))|$ contains a
smooth member but $\frac{-3}{2}(\omega_Y(-1))$ is not base--point--free, then $T$ is
the fixed part of $|\frac{-3}{2}(\omega_Y(-1))|$ and does not intersect the mobile
(base--point--free) part of $|\frac{-3}{2}(\omega_Y(-1))|$. This happens if and only if
$b-ae = \frac{1}{3}e-2$. On the other hand $\frac{-3}{2}(\omega_Y(-1))$ is
base--point--free if and only if $b-ae \geq e - 2$, which always holds if $0
\leq e  \leq 3$, since $b-ae \geq 1$.

\smallskip

\noindent Now we see that we can construct a triple cover of $Y$. If $b-ae \geq e - 2$, then $\frac{-3}{2}(\omega_Y(-1))$ is
base--point--free and obviously non trivial, so $\frac{-3}{2}(\omega_Y(-1))$ contains
smooth members. If $b-ae = \frac{1}{3}e-2$, then  $T$ is
the fixed part of $|\frac{-3}{2}(\omega_Y(-1))|$ and does not intersect the mobile
(base--point--free) part of $|\frac{-3}{2}(\omega_Y(-1))|$, so $|\frac{-3}{2}(\omega_Y(-1))|$
contains smooth (non--connected) members. Thus in all cases we can choose a
smooth divisor $B \in |\frac{-3}{2}(\omega_Y(-1))|$. Let $\pi: X \longrightarrow Y$ be the
triple cyclic cover of $Y$ branched along $B$. Since $B$ is smooth, so is $X$. Since $B
\in |\frac{-3}{2}(\omega_Y(-1))|$, ramification formula implies that
$\omega_X=\pi^*(\mathcal O_Y(1))$.

\noindent
Since $p_g(Y)=0$, then $\pi^*H^0(\mathcal O_Y(1))$ equals the complete series $H^0(\omega_X)$ (so $\varphi= i \circ \pi$) if, and only if, $H^0(\mathcal O_Y(1) \otimes \frac{1}{2}(\omega_Y(-1)))=0$.

\noindent
Since $\mathcal O_Y(1) \otimes \frac{1}{2}(\omega_Y(-1))=\frac{a-2}{2}T+\frac{b-e-2}{2}F$ and $a \geq 2$ (for $a$ is even), then we see that $H^0(\mathcal O_Y(1) \otimes \frac{1}{2}(\omega_Y(-1)))$ does not vanish, since $b-e-2 \geq 0$ (for $b \geq ae+1$ and $b-e$ is even).

\smallskip
\noindent
In the case $Y=\mathbb P^2$ it is straightforward to see that triple covers exist, since for $\mathcal O_Y(1)=\mathcal O_{\mathbb P^2}(d)$ with $d\geq 1$ odd, the linear series of $H^0(\frac{-3}{2}(\omega_Y(-1)))$ is very ample (so it contains smooth members), but $H^0(\mathcal O_Y(1) \otimes \frac{1}{2}(\omega_Y(-1)))=H^0(\mathcal O_{\mathbb P^m}(\frac{d-3}{2}))$ does not vanish, unless $d=1$.

\smallskip
\noindent
In the case $Y=\mathbb Q_2$ it is also straightforward to see that triple covers exist, since for $\mathcal O_Y(1)=\mathcal O_{\mathbb Q_2}(a, b)$ with $a, b \geq 2$ even, the linear series of $H^0(\frac{-3}{2}(\omega_Y(-1)))$ is very ample (so it contains smooth members), but $H^0(\mathcal O_Y(1) \otimes \frac{1}{2}(\omega_Y(-1)))=H^0(\mathcal O_{\mathbb Q_2}(\frac{a-2}{2}, \frac{b-2}{2}))$ does not vanish.
\end{proof}

\noindent
The next Propositions~\ref{ext=0}, \ref{2.ext=0}, \ref{3.ext=0} are crucial technical results that lead to the main Theorem~\ref{main} of the article. They deal with the vanishing of a certain Ext groups. Together with duality theorems this in turn proves first cohomology vanishing of the tangent sheaf twisted with relevant line bundles, which in turn leads to the main result. One has to be careful about the interpretations of the results, which is also a crucial part of the process.

\begin{proposition}\label{ext=0}
Let $Y$ be a smooth projective bundle on $\mathbf P^1$ of dimension $m \geq 3$ embedded, as a non--degenerate variety, in $\mathbf P^N$ by a complete linear series. Then $\mathrm{Ext}^1(\Omega_Y, \frac{1}{2}(\omega_Y(-1)))=0$ and $\mathrm{Ext}^1(\Omega_Y, \omega_Y(-1))=0$.
\end{proposition}
\begin{proof}
We already now by \cite[Proposition 1.2]{DCMDF} that Ext$^1(\Omega_Y,\omega_Y(-1))=0$. What we have to prove now is $\mathrm{Ext}^1(\Omega_Y, \frac{1}{2}(\omega_Y(-1)))=0$.

\noindent 
Since $\mathrm{Ext}^1(\Omega_Y, \frac{1}{2}(\omega_Y(-1)))=
H^1(\mathcal T_Y \otimes \frac{1}{2}(\omega_Y(-1)))$. We will prove $H^1(\mathcal T_Y \otimes \frac{1}{2}(\omega_Y(-1)))=0$.

\smallskip
\noindent
Let $H$ denote the hyperplane section. 
Having in account the sequence
  \begin{equation*}
 0 \longrightarrow \mathcal T_{Y/\mathbf P^1} \longrightarrow \mathcal T_Y
\longrightarrow  p^*\mathcal T_{\mathbf
P^1} \longrightarrow 0
\end{equation*}
and the dual of the relative Euler sequence
\begin{equation*}
 0 \longrightarrow \mathcal O_{\mathbf P(E)} \longrightarrow p^*E^* \otimes
\mathcal O_{\mathbf P(E)}(1) \longrightarrow  \mathcal T_{Y/\mathbf P^1}
\longrightarrow 0, 
\end{equation*}
we only need to check the vanishings of 
$H^1(p^*\mathcal T_{\mathbf P^1} \otimes \frac{1}{2}\omega_Y \otimes \mathcal O_Y(-\frac{1}{2}H))$, $H^1(p^*E^* \otimes \mathcal O_{\mathbf P(E)}(1) \otimes \frac{1}{2}\omega_Y \otimes \mathcal O_Y(-\frac{1}{2}H)))$ and $H^2(\frac{1}{2}\omega_Y \otimes \mathcal O_Y(-\frac{1}{2}H))$. By Serre duality this is equivalent to proving the vanishings of  $H^{m-1}(p^*\omega_{\mathbf P^1} \otimes \frac{1}{2}\omega_Y \otimes \mathcal O_Y(\frac{1}{2}H))$, $H^{m-1}(p^*E \otimes \mathcal O_{\mathbf P(E)}(-1) \otimes \frac{1}{2} \omega_Y \otimes  \mathcal O_Y(\frac{1}{2}H))$ and $H^{m-2}(\frac{1}{2}\omega_Y \otimes \mathcal O_Y(\frac{1}{2}H))$.
These cohomology groups are isomorphic to the cohomology groups of the push--down
to $\mathbf P^1$ of the three bundles. Indeed, in the three cases, when we
restrict the bundles to the fiber $F$ ($F =  \mathbf P^{m-1}$) we obtain line
bundles or direct sum of line bundles, so the intermediate cohomology of the
restriction of the bundles to $F$ vanish. The topmost (i.e.,
$H^{m-1}$) cohomology also vanishes since the line bundles just mention have
degree greater than or equal to $-\frac{m}{2}$. Now the three cohomology groups of the
push--downs on $\mathbf P^1$ are $0$ by dimension reasons except the
group $H^{m-2}(p_*(\frac{1}{2}\omega_Y \otimes \mathcal
O_Y(\frac{1}{2}H)))$ when $m = 3$. In this case we do also show that 
$H^{1}(p_*(\frac{1}{2}\omega_Y \otimes \mathcal
O_Y(\frac{1}{2}H)))=0$. Let $E = \mathcal O_{\mathbf P^1} \oplus \mathcal O_{\mathbf
P^1}(-e_1) \oplus \mathcal O_{\mathbf
P^1}(-e_2)$, with $0 \leq e_1 \leq e_2$, and let $T$ be the divisor on $Y$ such
that $\mathcal O_Y(T) = \mathcal O_{\mathbf P(E)}(1)$.  Recall that the
canonical divisor of $Y$ is 
$-mT-(e_1+e_2+2)F$ and $H$ is linearly equivalent to $aT + bF$. Since $H$
is very ample, $a > 0$ and $b > ae_2$. Then $\frac{1}{2}\omega_Y \otimes \mathcal
O_Y(\frac{1}{2}H)$ is linearly equivalent to $\frac{a-3}{2}T+ \frac{b-e_1-e_2-2}{2}F$.
If $a=1$, the vanishing is clear because the push--down bundle is $0$. If $a
\geq 3$ (and odd), the push--down bundle is a direct sum of line bundles and
the smallest of the degrees of these line bundles is $d=\frac{-(a-3)e_2 +
b-e_1-e_2-2}{2}$. Because of $b > ae_2$, $d > \frac{2e_2-e_1-2}{2}$. Since $e_1
\leq e_2$, $d > \frac{e_2-2}{2}$. Finally, since $e_2 \geq 0$, $d > -1$, so
$H^1$ of all the line bundles vanishes.
\end{proof}

\begin{proposition}\label{2.ext=0}
Let $Y=\mathbb P^m$, with $m \geq 3$, embedded, as a non--degenerate variety, in $\mathbf P^N$ by the complete linear series of $\mathcal O_{\mathbb P^m}(d)$, for which $d \geq 1$ and $m+d$ is odd. Then $\mathrm{Ext}^1(\Omega_Y, \frac{1}{2}(\omega_Y(-1)))=0$ and $\mathrm{Ext}^1(\Omega_Y, \omega_Y(-1))=0$. 
\end{proposition}
\begin{proof}
The proof is straightforward using the Euler sequence in $\mathbb P^m$. 
\end{proof}

\begin{proposition}\label{3.ext=0}
Let $Y=\mathbb Q_m$, with $m \geq 3$, embedded, as a non--degenerate variety, in $\mathbf P^N$ by the complete linear series of $\mathcal O_{\mathbb Q_m}(d)$, for which $ d \geq 1$ and $m+d$ is odd. Then $\mathrm{Ext}^1(\Omega_Y, \omega_Y(-1))=0$ and, if $m+d \geq 6$, then also $\mathrm{Ext}^1(\Omega_Y, \frac{1}{2}(\omega_Y(-1)))=0$.
\end{proposition}
\begin{proof}
The proof is also straightforward using the Euler sequence in $\mathbb P^{m+1}$ restricted to $\mathbb Q_m$ and the normal sequence of $\mathbb Q_m$ in $\mathbb P^{m+1}$. 
\end{proof}

\noindent
The previous steps yield the main Theorem.

\begin{theorem}\label{main}
Let $X, \varphi$ and $Y$ be as in the set--up. Then any general deformation of $\varphi$ is a canonical map and a finite morphism of degree $3$ onto its image, which is a
deformation of $Y$.
\end{theorem}

\begin{proof}
We are assuming $\varphi : X \to \mathbf P^N$ is the canonical map of $X$ which factors $\varphi =i \circ \pi$, for which $\pi: X \to Y$ is a triple cover onto $Y$, which fits in one of the three cases in our set--up  and the embedding $i: Y \hookrightarrow \mathbf P^N$ is non--degenerate and given by a complete linear series as in the set--up.
By Propositions~\ref{ext=0}, \ref{2.ext=0}, \ref{3.ext=0}, we know that $H^1(\mathcal T_Y \otimes \mathcal E)=0$.

\noindent
Let $(\mathfrak X,\Phi)$ be a deformation of $(X,\varphi)$ over a smooth curve $T$ and let
$(X',\varphi')$ be a general member of $(\mathfrak X,\Phi)$.

\noindent
Since $H^1(\mathcal T_Y \otimes \mathcal  E)=0$, then \cite[Corollary 1.11]{Wehler}  (or \cite[Theorem 8.1]{Hor3}) tells us that
$\mathfrak X$ can be realized as a covering $\Pi : \mathfrak X \to \mathfrak Y $ (which, after shrinking $T$, is finite, surjective and of degree $3$) of a deformation $\mathfrak Y$ of $Y$.
Now, since $H^1(\mathcal T_{\mathbf P^N}|_Y)=0$ (as can be easily checked in all cases of our set--up), any deformation $\mathfrak Y$ of $Y$ can be realized as a deformation $(\mathfrak Y, \mathfrak i)$ of $(Y,i)$ (see \cite[Theorem 8.1]{Hor3}). Since $i$ is an embedding, we can assume, after shrinking $T$, that $\mathfrak i$ is a relative embedding of $\mathfrak Y$ in $\mathbf P^N_T$.
Then $\mathfrak i \circ \Pi=\Phi'$ is a deformation of
$\varphi$. Now we saw in \cite[Lemma 2.4]{MP} that the only deformation of $\varphi$ is the (relative) canonical morphism, so $\varphi'$ is the canonical morphism
of $X'$ and, after shrinking $T$, we have $\Phi=\mathfrak i \circ \Pi$. 
\end{proof}

\section{Examples and Further remarks}
Let $\pi: X \longrightarrow Y$ be a cyclic triple cover of a variety of minimal degree $Y$. In \cite{Seshadri} it is shown that if the dimension of $X$ is odd, then the degree of $Y$ cannot be an even integer. Now assume that the dimension of $X$ is even, it is shown in \cite{Seshadri} that there are no canonical, generically finite morphisms of degree 3 whose image is a smooth rational scroll. 

\noindent
It is good to know what kind of triple covers does occur in mathematical nature. In \cite{Seshadri} it is shown that triple canonical covers of $\mathbf P^m$, a hyperquadric and a smooth rational normal scroll do actually occur in many cases. We construct the covers below (see Remark \ref{numerical}). 

\medskip 
\noindent
 To construct examples of canonical morphisms $\varphi: X \longrightarrow \mathbf P^N$ of degree $3$, such that $\varphi$ induces a triple cover onto a smooth rational normal scroll $Y$ of odd dimension $m\geq 3$ and even degree $r$, we carry out the following construction: let  $Y=\mathbf P(E)$  over $\mathbf P^1$, where $E=\mathcal O_{\mathbb P^1} \oplus \mathcal O_{\mathbb P^1}(-e_1) \oplus \cdots \oplus \mathcal O_{\mathbb P^1}(-e_{m-1})$, with $0 \leq e_1\leq \cdots \leq e_{m-1}$, let $T$ be the divisor on $Y$ such that $\mathcal O_Y(T) = \mathcal O_{\mathbf P(E)}(1)$ and let $H=T+bF$ be the hyperplane divisor on $Y$, where $F$ is the class of a fiber and $\; b-e_{m-1} \geq 1$. Then take the triple cyclic cover $\pi: X \longrightarrow Y$ branched along a smooth divisor linearly equivalent to $3(\frac {m+1}{2} H- \frac{r-2}{2}F)$ (a sufficient condition for such a smooth divisor to exist is $(m+1)(b-e_{m-1})\geq r-2$) . In this case $\omega_X=\pi^{\ast}\mathcal O_{Y}(H)$ and $H^0(\omega_X)=\pi^{\ast}H^0(\mathcal O_{Y} (H))$, which shows that $\varphi=i \circ \pi$, where $i: Y \hookrightarrow \mathbf P^N$ is the embedding, as wanted. 

\medskip 
\noindent
To construct examples of canonical morphisms $ \varphi: X \longrightarrow \mathbf P^N$ of degree $3$, such that $\varphi$ induces a triple cover onto $Y=\mathbf P^m$, with $m \geq 3$, take the triple cyclic cover $\pi: X \longrightarrow \mathbf P^m$ branched along a smooth divisor of degree $\frac {3(m+d+1)} {2}$, for which $m+d$ is odd and $1 \leq d \leq m$. In this case $\omega_X=\pi^{\ast}\mathcal O_{\mathbf P^m}(d)$ and $H^0(\omega_X)=\pi^{\ast}H^0(\mathcal O_{\mathbf P^m} (d))$, which shows that $\varphi=i \circ \pi$, where $i: \mathbf P^m \hookrightarrow \mathbf P^N$ is the embedding, as wanted.

\medskip 
\noindent
To construct examples of canonical morphisms $ \varphi: X \longrightarrow \mathbf P^N$ of degree $3$, such that $\varphi$ induces a triple cover onto a hyperquadric $Y=\mathbf Q_m$, with $m \geq 3$, take the triple cyclic cover $\pi: X \longrightarrow \mathbf Q_m$ branched along a smooth divisor in $\mathbf Q_m$ which is the complete intersection of $\mathbf Q_m$ and a hypersurface in $\mathbf P^{m+1}$ of degree $\frac {3(m+d-1)}{2}$, for which $m+d$ is odd and $1 \leq d \leq m-2$.  In this case $\omega_X=\pi^{\ast}\mathcal O_{\mathbf Q_m}(d)$ and $H^0(\omega_X)=\pi^{\ast}H^0(\mathcal O_{\mathbf Q_m} (d))$, which shows that $\varphi=i \circ \pi$, where $i: \mathbf Q_m \hookrightarrow \mathbf P^N$ is the embedding, as wanted.

\end{document}